\documentclass[11pt]{article}
\usepackage{amsmath,amsthm,amsfonts,amscd,amssymb,mathrsfs,cite}

\newtheorem{theorem}{Theorem}[section]

\newtheorem{lemma}[theorem]{Lemma}

\newtheorem{prop}[theorem]{Proposition}

\newtheorem{cor}[theorem]{Corollary}

\theoremstyle{definition}

\theoremstyle{remark}

\newcommand{\RR}{\mathbb R}

\title{A Multiple Parameter Linear Scale-Space for one dimensional Signal Classification}

\author{Leon A. Luxemburg, \thanks{Department of Mathematics, Texas A\& M University at Galveston, Texas, U.S.A.: luxembul@tamug.edu}, 
and Steven B. Damelin \thanks{Department of Mathematics, University of Michigan, 530 Church Street, Ann Arbor, MI 48109, USA: damelin@umich.edu}}

\begin{document}
\maketitle

\begin{abstract}
In this article we construct a maximal set of kernels for a multi-parameter linear scale-space that allow us to construct trees for classification and recognition of one-dimensional 
continuous signals similar the Gaussian linear scale-space approach. Fourier transform 
formulas are provided and used for quick and efficient computations. A number of useful 
properties of the maximal set of kernels are derived. We also strengthen and generalize 
some previous results on the classification of Gaussian kernels. Finally, a new topologically 
invariant method of constructing trees is introduced. 
\end{abstract}

\section{Introduction}
Scale-space filtering provides a powerful framework for the structural feature 
extraction, and classification and recognition of waveforms. It is based on 
convolving a signal with a one-parametric family of kernels and the convo- 
lutions can be used to construct certain trees to correspond to the original 
signal (\cite{k2y, k20y, k10y, k18y, k16y}).
In this article we solve the following important problems: 

\noindent (I) We construct a maximal set of kernels that allows us to construct trees 
and have the property that the signals with the same shape result in equivalent trees. It turns out that this maximal set of kernels is a set of $p$th frac
tional derivatives of a Gaussian. (We introduce a variety of these fractional 
derivatives in this and in subsequent articles). Our main result, stated at 
the beginning of Section 4, strengthens and generalizes the results obtained 
in \cite{k2y} and \cite{k20y}, and demonstrates an alternative for linear scale-space representation of signals beyond the widely published Gaussian method (\cite{k20y, k7y, k2y, k13y, k6y}, and more recent methods 
using a spectrum of families of kernels (\cite{k12y, k8y, k17y, k19y, k5y} ). 

\noindent (II) In this section we introduce a new, topologically invariant way of constructing trees, which we expect to be less sensitive to noise. 

\noindent (III) In Section 5 we calculate the Fourier transforms of the kernels, which 
gives us a quick and efficient computational procedure for constructing trees. 

\noindent (IV) In Section 2 we review the scale-space axioms and explained how they 
are incorporated into our formulas. Briefly, we base our scale-space on the 
following axioms: \\
\noindent
(*)  translation invariance \\
(*)  linearity \\
(*)  strong causality \\
(*)  scale invariance \\
The axioms we do not incorporate are the \\
(*)  semi-group axiom \\ 
(*)  preservation of positivity\\
\noindent (V) In Section $5$ we prove a number of properties concerning the convergence 
of functions $\phi (x, y)$ and their partial derivatives, where $\phi (x, y)$ is defined as 
a convolution of a signal $f (x)$ with a kernel $\Lambda_{\alpha \beta p}  (x, y)$. We develop a general 
theory of convergence of kernels. 

In the next article \cite{k9y} we continue as follows: 

\noindent (VI) Using results in (I) and (II)
above, we give a complete topological description of level curves $\phi_{x} (x, y) = c $ as well as level curves of arbitrary partial derivatives of $\phi (x, y)$. 

\noindent (VII) If we consider a convolution $\Lambda_{\alpha \beta p}  (x, y) = f (x) * \Lambda_{ \alpha \beta p} (x, y)$ of a signal  $f (x)$ with a kernel $\Lambda_{\alpha \beta p} (x, y)$ then as $\alpha , \beta$  and $p$ vary, this results in different zero crossings and different trees. We study bifurcations of the curves under parameter change. This gives us useful integer invariants for the signal which provides another way to classify the signals. We also use methods of differential and algebraic topology to derive some global formulas for tree invariants.

Let 
\begin{equation} G(x, \rho ) = \frac{\rho}  {\sqrt{2 \pi}}  e^{ -(x\rho )^{2} /2 }
,\hspace{3em} \rho  > 0 ,
\end{equation}
be a family of Gaussian kernels with parameter $\rho$  varying on $(0, \infty )$. For every $\rho  > 0$ and signal $f (x)$, one can consider a convolution 

\[
\phi (x, \rho ) = f (x) * G(x, \rho ) \equiv \int_{-\infty}^{\infty } f (x - v) G(v, \rho ) dv =
\int_{-\infty}^{\infty} f (v)G(x - v, \rho ) dv \;, 
\]
and on the half plane $(x, \rho )$, $\rho  > 0$ we construct level curves $\phi (x, \rho ) = 0$ for 
fixed $\rho$. It turns out that Gaussian kernels provide a certain monotonicity 
condition which guarantees that as $\rho$  increases, maxima of $\phi (x, \rho )$ increase 
and minima decrease. (In many references this property is called \emph{causality} 
\cite{k8y} or \emph{strong causality} \cite{k5y}, and may be regarded as an axiom of scale-space 
theory.) This, in turn, guarantees that contours $\phi (x, \rho ) = 0$ can close above 
but not below. A.P. Witkin has proposed to construct the trees in the 
following way in \cite{k2y, k20y}: a horizontal line is drawn touching the highest 
contour at its apex. The region between the x-axis and the horizontal line is 
then divided by the highest contour into three sub-regions which correspond 
to three branches of the first rank in the tree. As we proceed by drawing 
horizontal lines through the next highest apex of a contour, we get further 
subdivision of a sub-region into three parts, which gives us the branches of 
the second rank, etc., resulting in a trinary tree. The order of the branches 
of the same rank cannot be changed. The advantage of this method is 
that it is hierarchical, i.e., if three branches come out of a certain node 
(horizontal line), there would be no additions (of branches) to this node 
as $\rho$  increases. We note that this method may distinguish between two 
topologically equivalent sets of contours. (Two curves $\phi_{ 1}$ and $\phi_{ 2}$ on our half-plane can be called topologically equivalent if there is a homeomorphism of the closed half plane onto itself such that $\phi_{ 1}$ maps onto $\phi_{ 2}$.)

In the present paper we propose a new method for constructing a tree 
which can be used as an alternative to the very useful method described 
above. First, we form a node and construct first rank branches corresponding 
to connected contours that are not inside other contours. Then from each 
new node $n$, we construct branches corresponding to contours lying inside 
the contour corresponding to the node $n$. This forms branches of the second 
rank, etc. We expect these trees (which need not be trinary) to be less 
susceptible to noise --- which would be unlikely to disturb the topological 
equivalence of contours. For a function $f$ , the trinary tree constructed by 
the Witkin method (first method) will be denoted by $T W (f )$. The tree 
constructed via topological invariance (second construction) is denoted by 
$T T (f )$. It is easy to see that if $T W (f ) = T W (g)$ for two functions $f$ and $g$, 
then $T T (f ) = T T (g)$. The converse is not true in general. 

Constructing a tree $T T (f )$ or $T W (f )$ is a first important step in signal 
classification. It is clear, however, that a class of functions having the same 
tree is very broad. In some instances we would like to have a quick and rather 
coarse classification, but in many cases we want to separate various signals 
whose trees are identical. Consider, for example, a class $A_{n}$ of functions 
having no more than $n$ extrema. This class contains all polynomials of 
degree at most $n + 1$. Clearly, $A_{n}$ has an uncountable number of elements. 
However, it is easy to prove that there could be only a finite number of trees 
of both kinds constructed for the functions of this class. As a further step in 
refining this classification one might suggest using partial derivatives 
$\frac{\partial^{n} G }{\partial  x^{n}}$
of a Gaussian as convolution kernels. It turns out that for a rather large class 
of signals (containing all those that grow not faster than an exponential) 
one can prove that 
\[f (x) 
* \frac{\partial^{n} G }{\partial  x^{n}} = (-1)^{n} 
\frac{\partial^{n} f }{\partial  x^{n}}*G
.\] 
Thus, convolving with an $n$th partial derivative of $G$ in constructing a tree is 
the same as constructing a tree 
$T W\left( f ^{(n)} (x) \right)$ or $T T\left( f ^{(n)} (x) \right)$,
where $f^{ (n)} (x)$ 
is the $n$th derivative of $f (x)$. One might, therefore, construct a series of 
trees  
$TW\left(f^{(n)}(x) \right)$
or $TT\left(f^{(n)}(x) \right)$
for every function $f$ , thus providing a 
finer classification. It is easy to see, however, that there still would be un- 
countable subsets of the class $A_{n}$ consisting of elements indistinguishable by 
this classification. This is a motivation for finding a further refinement for 
signal classification, which we will do by also computing fractional deriva- 
tives of a Gaussian and allowing more than one parameter for describing 
the family of kernels. In addition, we propose to construct (along with 
trees for curves representing zero crossings $\phi (x, \rho ) = 0$) trees for level curves 
$\phi (x, \rho ) = c$, $c \neq 0$. The advantage of such a construction is (as we will prove 
in \cite{k9y}) that under very general conditions these level crossings are bounded 
for $c\neq 0$ and therefore allow for construction of trees by methods described 
above. Using this fact we will, in \cite{k9y}, construct two-dimensional surfaces 
and use Morse Theory to construct useful integer invariants for the pattern 
recognition of signals. 
\section{Conditions for a Linear Scale-Space}
In this article we will consider a signal $f$ to be \emph{admissible} if $f \in  L_{2} (\RR)$.  
More general admissible signals could be considered, but this is 
sufficient for practical purposes. Whenever a reference is made to an $n$th 
derivative $f^{ (n)} (x)$ of an admissible signal, it is assumed that $f (x)$ is $n$ times 
differentiable. 

The family of kernels must be such that the construction of the tree 
$T W_{g} (f )$ should be possible, i.e. the contours should satisfy the monotonicity 
condition (or axiom of strong causality), which we express here (see \cite{k2y}) as 
\begin{equation}
\begin{array}{l}
\phi_{\rho}  (x, \rho ) \phi_{xx} (x, \rho ) < 0  
\text{ whenever either } 
\phi_{\rho}  (x, \rho ) \neq 0 
\text{ or }
\phi_{xx} (x, \rho )\neq 0, \\
\text{where } \phi (x, \rho ) = f (x) * g(x, \rho ), \text{ for } g \in  M,
\end{array}
\end{equation} 
\noindent 
which is equivalent to the requirement that the level curves $\phi (x, \rho ) = c$ never 
close below (even locally), thus enabling us to implement the construction 
of trees from the level curves. 

\noindent \textbf{Remark}: If we require only monotonicity for zero crossing curves we need 
a somewhat weaker condition: 

\begin{equation}
\begin{array}{l}
\phi_{\rho}  (x, \rho ) \phi_{xx} (x, \rho ) < 0  
\text{ whenever } 
\phi_{x}  (x, \rho ) = 0 
\text{ and either } \\
\phi_{\rho} (x, \rho ) \neq 0 \text{ or } \phi_{xx} (x, \rho ) \neq 0. 
\end{array}
\end{equation} 

\noindent \textbf{Remark}: 
We have expressed the function $g = g(x)$ as depending on a single 
parameter $\rho$  (i.e., $g = g(x, \rho )$). We will show below that, in fact, more parameters will be needed in order to completely describe the maximal family $M$. The parameter $\rho$  is sometimes referred to as the \emph{bandwidth} parameter (\cite{k2y} for example). The parameter ($\sigma$) which is the reciprocal of $\rho$  is usually called the \emph{scale} parameter (\cite{k20y} for example). If the scale parameter is 
used instead of the bandwidth parameter then the inequality in (2) changes 
direction. 

We now state another condition on the set $M$ : Suppose that a signal $f (x)$ is convolved with $g(x, \rho ) \in  M$ at a certain value $\rho  = \rho_{0}$ , and let us 
scale $f (x)$ up by a certain number $k$ to get the signal $f (kx)$. Since scaling 
does not change the fundamental properties of the signal, we should be 
able to recognize it as the same signal, i.e., to construct the same tree 
$TW_{g} \left( f (kx) \right) = T W_{g} \left( f (x) \right)$. Clearly, for each $\rho$  the shape of the convolution 
$f (kx) * g(kx, \rho )$ is the same as that of $f (x) * g(x, \rho )$. We want, however, to 
require that after scaling of the argument $x$ and convolving with $g(x, y)$ we 
get the same shape of the signal as without scaling. More precisely, there 
should be a real number $\tau = \tau (k, \rho )$ such that for any $\rho  > 0$, 
\[
f (kx) * g(x, \tau ) = cf (x) * g(x, \rho ) 
\]
for some constant $c = c(k, \rho )$. In particular, this equality is to hold for the 
$\delta$-function. This implies that 
\[
g(kx, \tau (k, \rho )) = cg(x, \rho ) . 
\]
Now recall that the main idea of the scale-space approach is that the scale 
becomes finer as the parameter $(\rho )$ increases, therefore for a fixed $k$, $\tau (k, \rho )$
should increase as $\rho$  increases. For the same reason, as $k$ increases, $\tau (k, \rho ) $
should decrease. Thus, 
\[
\frac{\partial  \tau (k, \rho ) } {\partial  k} < 0 , 
\hspace{2em}
\frac{\partial  \tau (k, \rho ) }
{\partial  \rho } > 0 . 
\]
Furthermore, since scaling $x$ by $k$ results in a signal with a $k$-times finer 
scale, we are motivated to define $\tau(k, \rho ) = \rho /k$. Thus, our next requirement 
is that 

\begin{equation}
g\left(kx,\frac{\rho}{k}\right)=c(k,\rho)g(x,\rho).
\end{equation}
This equation is an expression of the axiom of \emph{scale invariance}. Different 
expressions for this axiom are offered in \cite{k19y} 
(where it is required that  $g(kx, \rho') = g(x, \rho')$ for some $\rho'$), \cite{k8y}, and \cite{k5y}, for example.
It implies, in particular, that the tree constructed for $f (kx)$ is the same as the one constructed 
for a signal $f (x)$. 

In order to assure stability of a kernel and the existence of 
$f * g(x, \rho )$ for 
continuous functions 
$f \in  L_{p} (\RR (x))$ with $1 \leq p \leq \infty$, we need the requirement $g(x, \rho ) \in  L_{1} (\RR(x))$ for all $\rho  > 0$, i.e. 
\begin{equation}
\int_{-\infty}^{\infty } |g(x, \rho ) | dx < \infty  , \text{ for all } \rho  > 0.
\end{equation}
Finally, we want to satisfy the maximality condition of $M$. This condition 
can be stated in two different ways: 

\emph{Every kernel satisfying properties (2), (4) and (5) belongs to M},

\noindent or 

\emph{every kernel $g(x, \rho )$ satisfying properties (2), (4) and (5) has an equiva-
\indent lent kernel $l(x, \rho ) \in  M $}.

\noindent 
(Two kernels $g(x, \rho )$ and $l(x, \rho )$ are equivalent if and only if there exist two 
real numbers $c$ and $k$ such that  $cg(kx, \rho ) = l(x, \rho )$.) 
We give the name $MW$ to an alternative maximal set which is described 
in the same way as $M$ except that the strong monotonicity condition (2) is 
replaced by a weak one (3). Thus $MW$ is described by either 

\emph{Every kernel satisfying properties (3), (4) and (5) belongs to $MW$}, 

\noindent or
 
\emph{every kernel $g(x, \rho )$ satisfying properties (3), (4) and (5) has an equiva-
\indent lent kernel $l(x, \rho ) \in  M W$}. 
We consider kernels $g(kx, \rho )$ and $g(x, \rho )$ equivalent because, if the condi- 
tion (4) is satisfied, we can obtain the kernel $g(kx, \rho_{0} )$ up to a constant 
factor by scaling $\rho_{0}$ down by $k$ and taking $g(x, \rho_{0} /k)$. We will also con- 
sider kernel $cg(x,\rho)$
equivalent to $g(x, \rho )$
for any constant $c\neq 0$,
because 
obviously, $T W_{cg} (f ) = T W_{g} (cf ) = T W_{g} (f )$ (this means that the axiom of 
linearity, with respect to multiplication, is satisfied by a convolution oper- 
ator) and we do not want to distinguish between f (x) and cf (x). Thus, 
equivalent kernels always result in the same tree for the same function.
 
\noindent \textbf{Remark}: Notice that condition (2) makes sense only if the partial derivatives $\phi_{xx}$, $\phi_{\rho}$, and $\phi_{x}$ exist. This is guaranteed by the existence of partial 
derivatives $g_{xx}$ , $g_{x}$ , and $g_{\rho}$  of the kernel $g(x, \rho )$. We will therefore need to 
verify below that partial derivatives $g_{x}$, $g_{\rho}$  , and $g_{xx}$ exist and are continuous 
and that they belong to $L_{1} (\RR(x))$. 

\noindent \textbf{Remark}: We have assumed a number of axioms which are required for an 
axiomatic basis for a linear scale-space representation of signals. Explicitly mentioned thus far are the axioms of \emph{linearity} (with respect to multiplication), \emph{strong causality}, and \emph{scale invariance}. In addition, we have implicitly assumed the axiom of \emph{translation invariance}, which is satisfied by a convolution operator acting on a signal. All of these axioms, together with the axiom of \emph{recursivity} (also known as the semi-group axiom ): 
$\left(f(x)*
g(x,\rho_{1})\right)
*g(x, \rho_{2} ) 
=
f (x) * g(x, \rho_{1} + \rho_{2} )
$, are necessary for a linear integral operator on admissible signals to be a Gaussian convolution operator. (It has been shown in \cite{k12y} that these axioms are, however, not sufficient for the operator to be Gaussian convolution. In fact, a continuum of convolution operators are possible, with the Gaussian kernel $g(x, \rho ) = G(x, \rho )$, Cauchy kernel $g(x, \rho ) = \frac{\rho}{\pi (1+\rho^{2} x^{2} )}$ , and sinc-function $g(x) = \frac{sin(x)}{x} $ being the three basic (closed-form) kernel functions which are possible. In this article we are not going to assume the axiom of recursivity, but we will assume (or 
have assumed) the other mentioned axioms. The family of kernel functions 
we will derive will be expressed in terms of a certain class of hypergeometric functions, so we pause to review some properties of these functions and 
introduce the class of functions we will need. 

\section{Review of Hypergeometric Functions}
In order to derive our expression for the kernel functions $g(x, \rho )$, we need 
to review some of the properties of hypergeometric functions; in particular, 
confluent hypergeometric functions. Furthermore we will show how these 
hypergeometric functions are related to derivatives of the Gaussian function $G(x, \rho )$,
and also show how certain confluent hypergeometric functions 
can be expressed as solutions of a second order linear differential equation. 
After that we will derive expressions for the Fourier transforms of these 
hypergeometric functions. 

For any $a \in \RR$ and any positive integer $n$ denote by $(a)_{n}$ the \emph{rising factorial} $a(a + 1) \cdots (a + n - 1)$, and set $(a)_{0} = 1$. If $b \in  \RR$ is not a negative 
integer or zero then the confluent hypergeometric function $M (a, b, z )$, known 
as Kummer's function of the first kind, given by the series 
\begin{equation}
M (a, b, z) = \sum_{n=0}^{\infty} \frac{(a)_{n}}{(b)_{n}}\frac{z^{n}}{n!}
\end{equation}
is an entire function (the ratio convergence test guarantees convergence for 
all complex numbers $z$). For any real number $p$ we define an even function 
$\Lambda^{e}_{p} (x)$ and an odd function 
$\Lambda^{o}_{p}(x)$ as follows: 

\begin{align}
\Lambda^{e}_{p}(x) &= \frac{1}{\sqrt{2 \pi}} M \left( \frac{p+1}{2}, \frac{1}{2}, -\frac{x^{2}}{2}  \right) \notag \\
&=  \frac{1}{\sqrt{2 \pi}}  \sum_{n=0}^{\infty} \left[  \frac{x^{2n}}{(2n)!} (-1)^{n} \prod_{k=0}^{n-1} (p+2k +1)   \right] \notag \\
& = \frac{1}{\sqrt{2 \pi}} - \frac{(p+1)x^{2}}{2! \sqrt{2 \pi}} + \frac{(p+1)(p+3)x^{4}}{4! \sqrt{2 \pi}} - \cdots
\end{align}

\begin{align}
\Lambda^{o}_{p}(x) &= \frac{x}{\sqrt{2 \pi}} M \left( \frac{p+2}{2}, \frac{3}{2}, -\frac{x^{2}}{2}  \right) \notag \\
&=  \frac{1}{\sqrt{2 \pi}}  \sum_{n=0}^{\infty} \left[  \frac{x^{2n+1}}{(2n+1)!} (-1)^{n} \prod_{k=0}^{n-1} (p+2k)   \right] \notag \\
& = \frac{x}{\sqrt{2 \pi}} - \frac{(p+2)x^{3}}{3! \sqrt{2 \pi}} + \frac{(p+2)(p+4)x^{5}}{5! \sqrt{2 \pi}} - \cdots
\end{align}

Functions
$\Lambda^{\theta}_{p}(x)$, 
$\theta = e,o$  are entire for all $p$. The following properties are a consequence of the definitions (7) and (8). 
\begin{equation}
\Lambda^{e}_{0}(x) = \frac{e^{-x^{2}/2}}{\sqrt{2 \pi}}
.
\end{equation}
\begin{equation}
\frac{d \Lambda^{e}_{p}(x)}{dx} = - (p+1)\Lambda^{o}_{p+1}(x)
.
\end{equation}
\begin{equation}
\frac{d \Lambda^{o}_{p-1}(x)}{dx} = \Lambda^{e}_{p}(x)
.\end{equation}
Therefore, 
\begin{equation}
\frac{d^{2n} \Lambda^{e}_{p} (x)}{dx^{2n}} = 
(-1)^{n} (p + 1)(p + 3)\cdots(p + 2n - 1)\Lambda^{e}_{p+2n}(x) , 
\end{equation}
\begin{equation}
\frac{d^{2n} \Lambda^{o}_{p} (x)}{dx^{2n}} = 
(-1)^{n} (p + 2)(p + 4)\cdots(p + 2n)\Lambda^{o}_{p+2n}(x) .
\end{equation}
By setting $p = 0$ in these equations it follows that 
\begin{equation}
\frac{d^{2n} \Lambda^{e}_{0} (x)}{dx^{2n}} = 
\frac{(-1)^{n} (2n)! \Lambda^{e}_{2n}(x)}{2^{n}n!} , 
\end{equation}
\begin{equation}
\frac{d^{2n} \Lambda^{e}_{0} (x)}{dx^{2n}} = 
\frac{(-1)^{n+1} (2n+2)! \Lambda^{0}_{2n+1}(x)}{2^{n+1}(n+1)!}.
\end{equation}
For hypergeometric series, it is known \cite[page 504]{k1y} that if $z$ is complex 
and $R(z ) < 0$, then as $|z | \rightarrow \infty$  we have the equality 
\begin{equation}
M(a,b,z) = \frac{\Gamma(b)}{\Gamma(b-a)} (-z)^{-a} \left( 1 + O \left(|z|^{-1} \right) \right).
\end{equation}
Notice that the gamma function is not defined for integers less than or equal 
to zero. Therefore, (16) does not apply only when $b - a$ is zero or a negative 
integer. Equality (16) implies that for $x$ real 
\begin{equation}
\Lambda^{e}_{p}(x) = C_{1}x^{-p-1}  \left( 1 + O \left(|x|^{-2} \right) \right) \;\text{ as } x \rightarrow \infty,
\end{equation}
\begin{equation}
\Lambda^{o}_{p}(x) = C_{2}x^{-p-1}  \left( 1 + O \left(|x|^{-2} \right) \right) \;\text{ as } x \rightarrow \infty,
\end{equation}
where $C_{1}$ and $C_{2}$ are determined from (16). By the previous comment, 
formula (17) does not apply only if $p = 2n$ for any non-negative integer $n$, 
and formula (18) does not apply only if $p = 2n + 1$ for any non-negative 
integer $n$ (in particular, it does apply if $p = 0$). 

Since the Gaussian kernel and its derivatives are in $L_{1}(\RR)$ formulas (9), 
(14) and (15) imply that $\Lambda^{ e}_{2n} (x)$ and $\Lambda^{ o}_{2n+1}(x)$ are in $L_{1}(\RR)$ for integers $n \geq 0$. Also, it now follows from (17) and (18) that 
$\Lambda^{\theta}_{p} (x)$, $\theta = e, o$ are in $L_{1} (\RR)$
if $p > 0$, but not in $L_{1} (\RR)$ if $p < 0$. 
Furthermore $\Lambda ^{e}_{0} (x) \in  L_{1} (\RR)$;
however, $\Lambda^{o} _{0} (x)\not\in L_{1} (\RR)$. 
In addition, the following equalities are satisfied only if $ p > -1$ 
\begin{equation}
\lim_{x\rightarrow \infty}\Lambda^{\theta}_{p}(x) = \lim_{x\rightarrow -\infty}\Lambda^{\theta}_{p}(x), \hspace{1em} \theta= e,o. 
\end{equation}
(This is obviously true for the Gaussian and its derivatives.) 
Let us continue with our preliminary observations. For real numbers $a$ 
and $p$ consider a differential equation
\begin{equation}
\ddot{h}(x) + a^{2} x \dot{h}(x) + a^{2} p h(x) = 0 ,
\end{equation}
where
\begin{equation}
\ddot{h} = \frac{d^{2}h(x)}{dx^{2}}, \hspace{1em} \dot{h}(x) = \frac{dh(x)}{dx}.
\end{equation}
\begin{lemma}
Let $h(x)$ be an arbitrary even (respectively odd) solution of 
equation (20). Then $h(x)$ is given by 
\begin{equation}
h(x) = c \lambda^{e}_{p-1}(ax) \hspace{1em} \left(\text{respectively } h(x) = c \Lambda^{o}_{p-1}(ax) \right)
\end{equation}
where
\begin{equation}
c = \sqrt{2\pi} h(0) \hspace{1em}\left(\text{respectively } c = \sqrt{2 \pi} \dot{h}(0) \right).
\end{equation}
The general solution of equation (20) is a linear combination 
\begin{equation}
\alpha\Lambda^{e}_{p-1}(ax) + \beta \Lambda^{o}_{p-1}(ax).
\end{equation}
\end{lemma}
\begin{proof}
Since functions $\Lambda^{\theta}_{p}(x), \theta= e, o$ are entire for every $p$, we can differentiate each monomial separately. It can be verified by direct substitution of 
(7) and (8) into (20) that the function 
$c \Lambda^{ e}_{p-1} (ax)$ (respectively $c \Lambda^{o}_{p-1}(ax)$) 
satisfies (20) with initial conditions $\Lambda^{ e}_{p-1} (0) = c/\sqrt{2 \pi }$ , 
$\frac{d\Lambda^{e}_{p-1}(0)}{dx}= 0$, (respectively $\Lambda^{0}_{p-1} (0) = 0$ , 
$\frac{d\Lambda^{e}_{p-1}(0)}{dx}= c/\sqrt{2 \pi }$ ).
Now let $h(x)$ be an even (respectively odd) solution of (20) with initial conditions $h(0) = \alpha$ , 
$\dot{h}(0) = 0$ 
(respectively $h(0) = 0$, $\dot{h}(0) = \beta )$.
From the uniqueness of solutions of differential equation (20), it follows that $h(x) = \alpha \sqrt{2 \pi } \Lambda^{e}_{ p-1} (ax)$ (respectively  $h(x) = \beta \sqrt{2 \pi } \Lambda^{o}_{ p-1} (ax)$ ). Hence we get the expression (24) for the general 
solution of (20). \end{proof}

\begin{lemma} Let $h(x)$ be a function satisfying two differential equations 
\begin{align}
&\ddot{h}(x) + a^{2}_{1} x \dot{h}(x) + a^{2}_{1} (p_{1} + 1)h(x) = 0, \hspace{1em} a_{1} \neq 0,&\hspace{-4em}  \text{and} \\ 
&\ddot{h}(x) + a^{2}_{2} x \dot{h}(x) + a^{2}_{2}(p_{2} + 1)h(x) = 0, \hspace{1em} a_{2}\neq 0 . \end{align}
Then $a_{1} = a_{2}$ and $p_{1} = p_{2}$ if $h(x)$ is not identical ly zero. 
\end{lemma}
\begin{proof} By Lemma 3.1 and (25), (26) we have:
\begin{equation} 
h(x) = \alpha_{1} \Lambda^{ e}_{p_{1}} (a_{1} x) + \beta_{ 1} \Lambda^{ o}_{p_{1}} (a_{1} x) = \alpha_{ 2} \Lambda^{ e}_{p_{2}} (a_{2} x) + \beta_{ 2} \Lambda^{ o}_{p_{2}} (a_{2} x) . 
\end{equation}
From (7) and (8) it follows that
\begin{equation}
h(0) = \frac{\alpha_{1} }{\sqrt{2 \pi}}  = \frac{\alpha_{2} }{\sqrt{2 \pi}}, \hspace{2em}
\dot{h}(0) = \frac{\beta_{1} }{\sqrt{2 \pi}}  = \frac{\beta_{2} }{\sqrt{2 \pi}}.
\end{equation}
Thus, $\alpha_{ 1} = \alpha_{ 2}$, and $\beta_{ 1} = \beta_{ 2}$. 
By virtue of analyticity of the function $h(x)$ and the uniqueness of power 
series expansion, the coefficients of corresponding terms $x^{n}$ are equal in both 
expressions for $h(x)$ in (27). By equating these coefficients and using (6), 
(7) and (8) we obtain the desired result. \end{proof}

\begin{theorem} If $p > 0$, then the Fourier transforms $F^{\theta}_{p} (\omega)$ of $\Lambda^{\theta}_{p} (x), \theta = e, o$ 
exist and 
\begin{align}
&F^{e}_{p}(\omega) = c |\omega|^{p}e^{-\omega^{2}/2} = c_{1}(s^{p}+ (-s)^{p} ) e^{s^{2}/2}, \\
&F^{o}_{p}(\omega) = d]\;sign(\omega) i |\omega|^{p}e^{-\omega^{2}/2} = c_{1}(s^{p}- (-s)^{p} ) e^{s^{2}/2},
\end{align}
where $i^{2} = -1$, $c_{1}$ is some complex number, and $c, d\neq 0$ are some real numbers. Also, 
$F^{ e}_{0}(\omega) = c e^{-\omega^{2}/2}$. 
\end{theorem}
\begin{proof} The existence of the Fourier transform $\mathcal{F}$ is guaranteed by the absolute integrability (as noted above) of functions $\Lambda^{e}_{p} (x)$ for $p > 0$, and by their differentiability. 

Let $\Lambda^{e}_{p} (x) = h(x)$. Then, for every $x$ it satisfies the equation 
\begin{equation} 
\ddot{h}(x) + x \dot{h}(x) + (p + 1)h(x) = 0 .
\end{equation}
Formulas (17) and (18) guarantee the existence of Fourier transforms of each 
summand in (31), and its convergence to zero as 
$|x| \rightarrow \infty$. By applying the Fourier transform to (31), we get 
\begin{equation}
s^{2} H (s) -\frac{d} {ds} [sH (s)] + (p + 1)H (s) = 0 , 
\end{equation}
here $H (s) = \mathcal{F} [h(t)] $ and $s = i\omega$, $\omega$ real. 

Here we used the fact that if 
$\mathcal{F} (g) = G(s)$, then 
$\mathcal{F} [g(t)t] = - \frac{dG(s)}{ds }$
provided 
$\int^{\infty}_{-\infty}tg(t)e^{- st }dt$
 converges uniformly in $s$ and 
$\int^{\infty}_{-\infty}  g(t)e^{- st }dt$  converges 
(see [14, 11.55a]). In our case, $g(t) = \dot{h}(t)$ and the conditions on convergence 
follow from (17), (18) and from the fact that $p > 0$. From (32) it follows 
that 
$\frac{H'(s)}{H(s)} = s + p/s$. Integrating this along $i\omega$ gives us 
\begin{equation}
H(s) = 2 c_{1}s^{p}e^{s^{2}/2}
.
\end{equation}
The path of integration should not include $0$, since we have a singular point 
there. Thus, if $s_{0} = i\omega_{0} $, for any $\omega_{0} > 0$, formula (33) holds for $s = j\omega$ 
with $\omega > \omega_{0}$. Since $H (s)$ must be real because $h(t)$ is an even function, 
$H (s) = c \omega^{p}e^{-\omega^{2}/2}$
for some real $\omega > 0$ and $c$ real. Now, since 
$H (-s) = H (s)$ 
for an even $h(x)$, we have 
$H (s) = c |\omega|^{p}e^{-\omega^{2}/2}$
for $s = i\omega$,  $\omega < 0$. This proves 
(29), since $F^{ e}_{p} (\omega) = H (s)$, and (30) can be proved in an entirely similar 
fashion. The last statement follows from the fact that $\Lambda^{e}_{0} (x)$ is the Gaussian 
function. 
\end{proof}

Now we are ready to state our main result. 
\section{Multiple Parameter Scale-Space}
\begin{theorem}
If $\{ g(x, \rho ), \rho  > 0 \}$ is a one-parametric family of kernels 
satisfying the scaling property (4), the stability condition (5), and

\noindent i) the strong monotonicity condition (2) if and only if 

\noindent ii) the weak monotonicity condition (3) if 

\[g(x, \rho ) = \alpha \rho^{ p+1} \Lambda^{e}_{p} (ax\rho ) + \beta  \rho^{ p+1} \Lambda^{o}_{p} (ax\rho ) 
\]
for some real $\alpha , \beta  , a$ and $p > 0$ or (in case $p = 0$)
\[ 
g(x, \rho ) = \alpha \rho \Lambda^{e}_{0} (ax\rho ) = \alpha G(ax, \rho ) , 
\]
where $\Lambda^{e}_{p} (x)$ and $\Lambda^{o}_{p} (x)$ are the one-parametric families of entire functions 
defined in (7) and (8), which belong to $L_{1} (\RR)$ for the indicated values of the 
parameter $p$. 

Moreover, 
\begin{equation}
\frac{\partial^{2n} G(x, \rho )}{\partial  x^{2n}}
 = c_{2n} \rho^{2n+1} 
\Lambda^{e}_{2n} (x\rho ) , 
\end{equation}
\begin{equation}
\frac{\partial^{2n+1} G(x, \rho )}{\partial  x^{2n+1}}
 = c_{2n+1} \rho^{2n+2} 
\Lambda^{0}_{2n} (x\rho ) , 
\end{equation}
where $c_{2n} = 
\frac{(-1)^{n}(2n)!}{2^{n} n!}$ , $c_{2n+1} = c_{2(n+1)}$ , and $G(x, \rho )$ is the Gaussian kernel 
defined in (1). 
\end{theorem}
\noindent \textbf{Remark}: We are not requiring the axiom of recursivity in Theorem 4.1.

\medskip
Before giving the proof of Theorem 4.1, we first need: 
\begin{lemma} Let $\mu$ and $\nu$ be non-zero vectors in a vector space $V$ , over the 
field of complex numbers, with a scalar product $\langle \cdot,\cdot\rangle$.Then, the statement: 

for any vector $f \in  V$ such that either 
$\langle f, \mu \rangle \neq 0$ or $\langle f,\nu\rangle \neq0$, we have 
$\indent \langle f, \mu \rangle \cdot \langle f, \nu \rangle< 0 $

\noindent is true if and only if there is a negative number $\alpha$  such that $\mu = \alpha \nu$.
\end{lemma}
\begin{proof} Suppose that the statement is true, and $\mu = \alpha \nu + \zeta$ for some scalar $\alpha$  
and vector $\zeta \neq 0$ perpendicular to $\nu$.
 Then since $\langle  \mu,\zeta \rangle = \langle \zeta, \zeta \rangle \neq 0$ we have, 
by hypothesis, that
$
\langle \mu, \zeta \rangle \cdot \langle  \nu, \zeta \rangle < 0
$,
 which contradicts our assumption that 
$\nu$ is perpendicular to $\zeta$. Therefore, we must have $\mu = \alpha \nu$. Now let us take 
$f = \nu$. Obviously, $\langle f, \nu \rangle \neq 0$, and therefore,
$
\langle f, \nu \rangle \cdot \langle f, \mu \rangle = \alpha \langle \nu, \nu \rangle^{2} < 0
$
which implies that $\alpha  < 0$. The proof of the converse is trivial. \end{proof}

In the proof which follows we will assume, as a result of the scale invariance axiom, that $g(x, \rho )$ has the form 
\begin{equation}
g(x, \rho ) = \zeta (\rho )h(x\rho ) . 
\end{equation}
for some continuous $\zeta$ and $h$, so that the convolution of a horizontally scaled 
signal $f (kx)$ with the kernel $g(x, \rho )$ will be the same as the horizontally 
scaled convolution of the unscaled signal $f (x)$ with the kernel $g(x, \rho /k)$, 
times a constant. This assumption is more general that the assumption 
$g(x, \rho ) = \rho^{ 2} h(x\rho )$ made in \cite{k21y}, for instance, and slightly different from the 
assumption $g(x, \rho ) = \zeta(\rho )h(x\zeta (\rho ))$ made in \cite{k19y}. The reader will see below 
that our assumption is suitable for our purposes. 

\begin{proof}[Proof of Theorem 4.1] Let us write the expressions for $\phi (x, \rho )$ and its partial 
derivatives. 
\begin{align}
\phi(x,\rho) & = \zeta(\rho) \int^{\infty}_{-\infty} f(\nu) h(\rho(x-\nu))d\nu \notag \\
\phi_{x}(x,\rho) & = \zeta(\rho) \rho \int^{\infty}_{-\infty} f(\nu) \dot{h}(\rho(x-\nu))d\nu \\
\phi_{xx}(x,\rho) & = \zeta(\rho) \rho^{2} \int^{\infty}_{-\infty} f(\nu) \ddot{h}(\rho(x-\nu))d\nu \\
\phi_{p}(x,\rho) & =  \int^{\infty}_{-\infty} f(\nu) [ \dot{\zeta}(\rho)  h(\rho((x-\nu))+ \zeta(\rho) \dot{h}(\rho(x-\nu))]d\nu
\end{align}
By making the substitution $t = -\nu \rho$ , and defining $f_{1} (t) = f (\nu) = f (-t/\rho )$, 
he integrals (37), (38) and (39) can be expressed as follows: 
\begin{align}
\phi_{x}(x,\rho) & = \zeta(\rho) \int^{\infty}_{-\infty} f_{1}(t) \dot{h}(x\rho + t)dt\\
\phi_{xx}(x,\rho) & = \zeta(\rho) \rho \int^{\infty}_{-\infty} f_{1}(t) \ddot{h}(x\rho + t)dt \\
\phi_{p}(x,\rho) & =  \frac{1}{\rho} \int^{\infty}_{-\infty}  [ \dot{\zeta}(\rho)  h(x \rho + t)+ \zeta(\rho)\left(x + \frac{t}{\rho} \right) \dot{h}(x \rho + t)] f_{1}(t)dt
\end{align}
We let $\tau(t) = x\rho  +t$, then the strong monotonicity condition (2) is equivalent 
to the statement

Whenever either $\langle f_{1}, \ddot{h} \circ \tau \rangle \neq 0$ or $\langle f_{1}, \dot{\zeta} h \circ \tau + ]\tau \left( \zeta(\rho)/\rho \right) \dot{h}\circ \tau \rangle \neq 0$
\begin{equation} \text{then} \hspace{1em} \langle f_{1}, \ddot{h}\circ \tau \rangle \langle f_{1} \dot{\zeta}(\rho) h\circ \tau + \tau \left( \zeta(\rho) /\rho \right)\dot{h}\circ \tau \rangle <0,
\end{equation}
where the scalar product of functions $f_{1}$ and $f_{2}$ in $L_{2} (\RR)$ is given by 
$\langle f_{1},f_{2} = \int^{\infty}_{-\infty}f_{1}(t)f_{2}(t) dt$. 
Let us denote $\mu = \dot{\zeta}(\rho)h\circ \tau + \tau \left(  \zeta(\rho) / \rho \right) \dot{h} \circ \tau$, and $\nu = \ddot{h}\circ \tau$.
Using this notation we may restate (43) as 
\begin{align}
\text{Whenever either }& \langle f_{1}, \mu \rangle \neq 0 \text{ \; or \; } \langle f_{1}, \nu \rangle \neq 0, \notag \\
\text{then }& \langle f_{1}, \mu \rangle \langle f_{1}, \nu \rangle < 0
\end{align}

From Lemma 4.2 this is equivalent to 
\begin{equation}
\nu \equiv - a^{2} \mu, \hspace{1em} \text{for some real number } a = a(\rho ),
\end{equation}
which is equivalent to the differential equation 
\begin{equation}
\ddot{h}(t) + a^{2}(\rho) \left[ \dot{\zeta}(\rho) h(t) + \frac{\zeta(\rho)}{\rho} t  \dot{h}(t)   \right] = 0.
\end{equation}
Equation (46) has to be satisfied for all $\rho  > 0$. From Lemma 3.2, it follows 
that 
\begin{equation}
\frac{a^{2}(\rho) \zeta(\rho)}{\rho} \text{ and } a^{2}(\rho)\dot{\zeta}(\rho) \text{ are constants.}
\end{equation}
Therefore, their ratio $\rho\dot{\zeta}(\rho )/\zeta(\rho )$ is a constant and there exists a number $p$ 
such that 
\begin{equation}
\frac{\rho \dot{\zeta}(\rho)}{\zeta(\rho)} = 1 + p \hspace{3em} (\text{or } \dot{\zeta}(\rho) = (1+ p) \frac{\zeta(\rho)}{\rho}).
\end{equation}
Solving this differential equation determines that 
\begin{equation}
\zeta (\rho ) = c \rho^{ p+1} ,
\end{equation}
for some constant $c \neq 0$. Since the coefficients of (46) are constants, there 
exists a constant $a \neq 0$ such that (46) can be rewritten as 
\begin{equation}
\ddot{h}(t) + a^{2} t \dot{h}(t) + a^{2} (p + 1)h(t) = 0 .
\end{equation}
By Lemma 3.1, h(t) is given by 
\begin{equation}
h(t) = \alpha_{ 1} \Lambda^{e}_{p} (at) + \beta_{ 1} \Lambda^{o}_{p} (at) , 
\end{equation}
and then from (36), (49) and (51) we find that the strong monotonicity 
condition (2) is equivalent to
\begin{equation}
g(x, \rho ) = \alpha  \rho^{ p+1} \Lambda^{e}_{p} (ax\rho ) + \beta  \rho^{ p+1} \Lambda^{o}_{p} (ax\rho ) ,
\end{equation}
where $\alpha  = c \alpha_{ 1} , \beta  = c \beta_{ 1}$. The necessity and sufficiency of condition (5) 
follows from the statements in Section 3, and (4) follows immediately from 
(52). In addition, it follows from (41), (42), (46) and (49) that 
\begin{equation}
\phi_{xx} (x, \rho ) = 
-\phi_{\rho}  (x, \rho )\rho^{3} .
\end{equation}
(This is another way to see that (2) is satisfied). 

It is clear that the the weak monotonicity condition (3) is satisfied by 
the given family of kernels.  \end{proof}
\noindent \textbf{Remark}: We will show below that $\Lambda^{o}_{p} (x)$ and $\Lambda^{e}_{p} (x)$ are in some sense $p$th fractional derivatives of a Gaussian. We have already shown this for integer 
(non-fractional) values of $p$ in (34). 

\noindent \textbf{Remark}: The reason we are interested in the weak monotonicity condition 
is that it guarantees a possibility of constructing trees based on zero crossing 
curves only. 

\noindent \textbf{Remark}: In the proof of Theorem 4.1 we only required the kernel functions 
h to have two derivatives in x and one derivative in $\rho$. However, the kernel 
functions $g(x, \rho )$ which we derive are in fact infinitely differentiable in $x$ and 
$\rho$. 

Theorem 4.1 describes a maximal family of kernels which allow us to con- 
struct trees via methods described above and thus provide us with a \emph{maximal} 
set of invariants that can be used to classify signals. Since $g(x, \rho )$ in Theorem 4.1 is in $L_{1} (\RR) \subset L_{2} (\RR)$ we can normalize so that
$
\int^{\infty}_{-\infty}|g(x,\rho)|^{2} dx = 1.
$
By means of Parseval's identity, this is equivalent to
$
\int^{\infty}_{-\infty}|\mathcal{F}(g)(s,\rho)|^{2} ds = 1.
$
Then by making use of the Fourier transform idenities in (29) and (30) 
we obtain the relation 
\begin{equation}
\alpha^{ 2} - \beta^{ 2} = \frac{1}{(a\rho )^{p} \Gamma(p + 1/2)} .
\end{equation}
This means that our family $M$ of kernels can be described by means of 
parameters $a, \rho , p$, and one coefficient paramter. In certain situations we 
may want to consider only kernel functions which have either even or odd 
symmetry. Hence we define 
\begin{align}
g^{e}_{p} (x, \rho ) := \rho^{p+1} \Lambda^{e}_{p} (x\rho ) ,\hspace{1em} p \geq 0, \\ 
g^{o}_{p} (x, \rho ) := \rho^{p+1} \Lambda^{o}_{p} (x\rho ) ,\hspace{1em} p > 0 ,
\end{align}
and we have the following: 
\begin{cor} Let $g(x, \rho ), \rho  > 0$ be a parametric family of kernels satisfying conditions (2), (4) and (5) above, and let $g$ be either even or odd with 
respect to $x$. Then there exist three real numbers $c$, $a$, and $p$ such that 
\begin{align}
g(x, \rho ) = c \rho^{p+1} \Lambda^{e}_{p} (ax\rho ) ,\hspace{1em} p \geq 0 ,  \hspace{1em} \text{if } g(x, \rho ) \text{ is even}. \\ 
g(x, \rho ) = c \rho^{p+1} \Lambda^{o}_{p} (ax\rho ) ,\hspace{1em} p > 0 ,  \hspace{1em} \text{if } g(x, \rho ) \text{ is odd}. 
\end{align}
\end{cor}

As we showed in Section 3 the kernels $g^{\theta}_{p} (x, \rho ) = \rho^{ p+1} \Lambda^{\theta}_{p} (x\rho ), \theta = e, o $ are not stable for $p < 0$. Moreover, for $p < -1$ the convolution $g^{\theta}_{p}*f $ may 
not exist for an admissible signal because $g_{p} (x, \rho )$ no longer converges to 
zero as $x \rightarrow \infty$ . However, if $f (x)$ is a transient signal, i.e. a signal which is 
zero outside of some finite interval $[-T , T ]$, then convolutions $g_{p}*f$ are 
well defined for any real value $p$ and any piecewise continuous signal $f$ 
(and even for any signal $f$ integrable on $[-T , T ]$). An exact analogue of 
Theorem 4.1 (without the stability condition (5) and without the restriction 
$p > 0$) therefore applies for transient signals. 

Let us discuss the one-parametric families of trees $TT^{e}_{p}(f )$ and $TW^{e}_{p}(f )$. 
We have seen before that $TT^{ e}_{ 2n} (f ) = T T^{ e}_{0} (f^{ (2n)} )$ and $T W^{ e}_{2n} (f ) = T W ^{e}_{0} (f ^{(2n)} )$. 
Since the trees for $f ^{(2n-2)}$ and for $f ^{(2n)}$ are obviously very different for most 
functions $f$ , there are points $p$ between $2n-2$  and $2n$ at which the structures 
of trees $TW^{e}_{p}(f )$ and $TT^{e}_{p}(f )$ change. In fact, trees $TT^{e}_{p}(f )$ and $TW^{e}_{p}(f )$, 
in general, depend continuously on $p$, and such discontinuity in structure 
occurs only for a discrete set of $p$'s (see \cite{k9y}). Also, see \cite{k4y} for interesting 
results in this direction. 

A more general approach can be taken if we consider contours 
$\frac{\partial^{k} \phi (x,\rho)}{\partial x^{k}} = c$
where $c$ is an arbitrary number and $\phi (x, \rho )$ is defined as a convolution of 
$f (x)$ and $\alpha \Lambda^{e}_{p} (x) + \beta  \Lambda^{o}_{p} (x)$. Obviously such contours depend on the parameters $\alpha , \beta  , c$ and $p$, and we will study the change in the corresponding tree 
structure as these parameters vary. We will use Morse Theory to study these 
bifurcations and also study many useful properties of the convolutions with 
our kernels and compute Fourier transforms of our kernels. This allows us 
to efficiently implement our methods on computers.

\section{Properties of Convolution Kernels}
Throughout this section we will suppose that any function $\chi^{\theta}$ (where $\chi$ may 
be any of $F, f , \Lambda , \phi , g$) is defined for $p 
\geq 0$ when $\theta = e$, and defined for $p > 0$ 
when $\theta = o$. We will also assume that $k, l, m, n$ are integers and  $p, q$ are 
nonnegative real numbers. 

As before, let $g^{theta}_p = \rho^{p+1} \Lambda^{\theta}_{p} (x\rho )$ where $\theta = e,o$ . Then, in order to be able to compute the convolution $\phi^{\theta}_{p} (x, \rho ) = f (x) *g (x, \rho)$ more efficiently, we need to know its Fourier transform. This will allow us to perform computations in 
the frequency domain. Another question closely related to this is a question 
of finding the limiting function $\phi^{\theta}_{p} (x, \rho )$ as $\rho  \rightarrow \infty $. It is well known that for the Gaussian $g^{e}_{ 0} (x, \rho )$, the Fourier transform is 
$ce^{s^{2} /\rho ^{2}}$ where $c$ is some constant. As $\rho \rightarrow \infty$, $ce^{s^{2} /\rho ^{2}} \rightarrow c$ ; thus, when we perform the convolution, 
we get the same signal $f (x)$ in the limit (this is due to the fact that the $\delta$
 - function has the Fourier transform equal to 1). Therefore, the limiting 
value for $\phi^{\theta}_{0} (x, \rho )$ as $\rho \rightarrow \infty$  is the function $f (x)$ itself, i.e., as $\rho$  increases, the complexity of the convolution $\phi^{\theta}_{0} (x, \rho )$ increases until we come to a signal $f (x) = \phi 0 (x, \infty )$ in the limit. It turns out that if $n$ is an integer, the limiting 
function for $\phi^{\theta}_{n} (x, \rho )$ as $\rho \rightarrow \infty$  is the $n$th order derivative of $f (x)$, where $\theta$ is odd or even, depending on whether $n$ is odd or even. 

We will generalize this simple property for every $p$ by finding the limit 
$\phi^{\theta}_{p} (x, \rho )$ as $\rho \rightarrow \infty$  , proving its existence and showing that it is in some sense a fractional pth order derivative of the signal. We will also consider 
here a number of very useful properties of convolutions with $g^{\theta}_p (x, \rho )$. 

\noindent \textbf{Remark}: Since the trees of the signals do not change if we scale it or 
multiply it by a constant, we will determine Fourier transforms only up to 
a constant factor and will not spend any time on determining these factors 
precisely. 

In what follows, a function $f (x)$ and its Fourier transform 
$\mathcal{F} (f )$ will be called a \emph{Fourier pair} if they are connected by the usual direct and inverse 
Fourier transform formulas. In this case, we will write $f \sim \mathcal{F} (f )$. In gen- 
eral, the direct Fourier transform formula does not imply the inverse, i.e., if 
$\mathcal{F} (f ) = F (s)$, it is not generally true that $f (x) = \mathcal{F}^{-1}\left(F(s)\right)$. 

Let us recall a definition of fractional derivatives. Let $q < 0$, then \cite{k6y} 
\begin{equation}
\frac{d^{q} f (x)}{ dx^{q}} = 
\frac{1} {\Gamma(-q)} \int^{x}_{0} \frac{f (\rho ) dy} {(x -\rho )^{q+1}} = 
\frac{1} {\Gamma(-q)} \int^{x}_{0} \frac{f (x-\rho ) dy} {\rho ^{q+1}}
\end{equation}
This definition of fractional derivatives belongs to Liouville. We propose a 
related but slightly different definition. Let 
$\mathcal{F} (f ) = F (s)$; then for $s = j\omega$,
\begin{equation}
\frac{d^{p} f (x)}{dx^{p}} = f^{(p)} (x) = 
\frac{c}{2\pi j} \int^{\infty}_{-\infty}  s^{p} F (s)e^{sx} ds
,\end{equation}
provided the integral converges. We put a constant real number in (60) as 
usual to emphasize that we are only interested in functions up to a scalar 
factor. In what follows, we will use the following simple fact from the theory 
of Fourier transforms (see \cite{k3y}: 

\begin{prop} If a function $f (x)$ is continuously differentiable with $f ^{(n)} (x) \in  L_{1} (\RR)$ for all $n, 0 \leq n \leq m$, then for some $c > 0$, 
\begin{equation}
\left| \mathcal{F}\left(f(x) \right)(s)\right| = |F(s)| \leq \frac{c}{|s^{m}|}
.\end{equation}
\end{prop}

We give a motivation for our definition as follows: let $\mu_{\alpha}(t) = t^{-\alpha }$ for $t > 0$ and 
$\mu_{\alpha}(t) = 0$ for $t \leq 0$ where $0 < \alpha  < 1$; then 
$\mathcal{F} \left( \mu_{\alpha}(t)\right)(s) = 
s^{\alpha -1} \Gamma(1-\alpha )$.
Therefore, by the convolution theorem we have (supposing certain convergence properties are satisfied for every $f (x)$ and $p$): 
\begin{align*}
f^{(p)}(x)  = \mathcal{F}^{-1}\left(cs^{p}F(s) \right) &= \\
 \mathcal{F}^{-1}\left( cs^{p} \right)*f(x) &= c \mu_{p+1} * f(x) = c \int_{0}^{\infty} \frac{f(x-\rho)}{\rho^{p+1}} dy
.\end{align*}

If, in addition, $f (\rho ) = 0$ for $\rho  < 0$ then $f^{ (p)} (x) = c \int_{0}^{\infty} \frac{f(x-\rho)}{\rho^{p+1}} dy$
and this coincides with the classical definition (59). 
We need to establish the range of applicability of definition (60) and 
show that it extends the usual definition of the derivative when p is an 
integer. Our purpose here is not to develop a fractional calculus, but to use 
some of its concepts to explain our results and their applications. As usual, 
we introduce the equivalence relation equating a function $f (x)$ with $cf (x)$ 
where $c$ is a constant not equal to zero. 
\begin{theorem}
If a function $f (x)$ satisfies the conditions of Proposition 5.1 
with $m > p + 1$, then 
\begin{itemize}
\item[(i)] formula (60) gives us a wel l defined $f^{ (p)} (x)$. 
\item[(ii)] Also, if $xf (x)$ is in $L_{1} (\RR)$, then 
\begin{equation}
cs^{p} F (s) = \mathcal{F}\left( f (p) (x)\right), \hspace{2em} \mathcal{F}\left(cs^{p}F(s)\right) = f ^{(p)} (x). 
\end{equation} 
\item[(iii)] Formula (60) gives us the usual derivative if $p$ is an integer. 
\item[(iv)] If $p + q + 1 < m$, then 
$\frac{d^{p}}{dx^{p}} \left(\frac{d^{q}}{dx^{q}} f (x)\right)= 
\frac{d^{p+q }}{dx^{(p+q)}} f (x).$ 
\end{itemize}
\end{theorem}

\begin{proof}
(i) Follows immediately from (61). 

\noindent (ii) If $xf (x)$ is in $L_{1} (\RR)$, then $F (s)$ is continuously differentiable (see \cite{k3y}). 
Therefore, $s^{p} F (s)$ is continuously differentiable. It is also in $L_{1} (\RR)$ as is seen 
from (61) and the condition $m > p + 1$. Since $2\pi i f^{ (p)} (-x) = \mathcal{F}\left(cs^{p} F (s)\right)$
by (60), we can now apply the theorem on the inversion of the Fourier transform \cite{k3y}, and conclude that 
$ \mathcal{F}\left(f^{(p)}(x) \right) = cs^{p} F (s)$ and $f ^{(p)} (x) = \mathcal{F}^{-1}\left(cs^{p} F (s)\right)$.

\noindent (iii) If p is an integer, the assertion of (iii) is a very well known fact. 

\noindent (iv) Follows immediately from (ii) and (60).
\end{proof}

\noindent \textbf{Remark}: Formulas (33) and Theorem 5.2 imply that the linear combination 
$h_{p} (x) := \alpha \Lambda^{e}_{p} (x) + \beta  \Lambda^{o}_{p} (x)$ is actually the $p$th fractional derivative of $c_{1} G(x)$, which is the Gaussian function multiplied by some complex constant $c_{1}$. 

It will turn out that the limit function for $f (x) *h(x\rho )$ as $\rho\rightarrow \infty$  is 
$f ^{(p)} (x)$. This is a very important generalization of the formula for $p = 0$. 
This property follows from the following theorem. Before we state Theo- 
rem 5.4, we need a preliminary lemma. 

\begin{lemma} Let $f \in  L_{2} (\RR)$ be a continuously differentiable function, then 
\begin{itemize}
\item[(a)] $\phi^{\theta}_{p} (x, \rho ) = f (x) * \rho ^{ p+1 }\Lambda^{e}_{p} (x\rho ),$ $\theta = e, o$, is well defined and infinitely differentiable in $x$ and $\rho$. 

\item[(b)] The following functions constitute Fourier pairs. 
\begin{eqnarray}
g^{e}_{p} (x, \rho ) = \rho^{ p+1} \Lambda^{e}_{p} (x\rho )\sim c(\rho ) \left(s^{p} + (-s) ^{p} \right)
e^{s ^{2} /(2\rho ^{2} ) }
,\\
g^{o}_{p} (x, \rho ) = \rho^{ p+1} \Lambda^{o}_{p} (x\rho )\sim c(\rho ) \left(s^{p} - (-s) ^{p} \right)
e^{s ^{2} /(2\rho ^{2} ) }
,\\ 
\phi^{e}_{p}(x,\rho) \sim  c(\rho ) \left(s^{p} + (-s) ^{p} \right) e^{s ^{2} /(2\rho ^{2} ) } F(s)
,\\
\phi^{o}_{p}(x,\rho) \sim  c(\rho ) \left(s^{p} - (-s) ^{p} \right) e^{s ^{2} /(2\rho ^{2} ) } F(s)
,\end{eqnarray} 
where $F (s) =\mathcal{F}\left( f(x) \right)$, $s = i\omega$, and $c(\rho )$ is a constant depending on $\rho$. 
\end{itemize}

\end{lemma}
\begin{proof}
(a) trivially follows from the theorem on differentiation of improper integrals by a parameter \cite{k14y,k15y}. (b) follows from Theorem 3.3. 
\end{proof}

\begin{theorem} If a function $f (x)$ satisfies the conditions of Proposition 5.1 with $m > p + 1$, then 
\begin{itemize}
\item[(i)] Functions $\phi^{e}_{p} (x, \rho )$, $\theta = e, o$ converge uniformly as $\rho \rightarrow \infty$  to a continuous function $f^{\theta}_{p} (x)$ in $C^{ k} (\RR)$ if $k < m - p - 1$. 
\item[(ii)] If $xf (x) \in  L_{1} (\RR)$, then for any $k, 0 < k < m - p - 1$ functions 
$\frac{d^{k}} {\phi^{ e}_{p} (x,\rho )}{dx^{k}}$, $\theta = e, o$ converge uniformly to 
$\frac{d^{k}} {dx^{k}} f^{\theta} _{p} (x)$ as $\rho \rightarrow \infty$. 

\item[(iii)] For the functions $f^{\theta}_{p} (x)$, $\theta = e, o$ we have the following Fourier 
correspondence: 
\begin{eqnarray*}
f^{e}_{p})(x) \sim  c(\rho ) \left(s^{p} + (-s) ^{p} \right) F(s)
,\\
f^{o}_{p}(x) \sim  c(\rho ) \left(s^{p} - (-s) ^{p} \right) F(s),
\end{eqnarray*}
where $F (s) = \mathcal{F}\left(f (x)\right)$.
 
\item[(iv)] Functions $\phi^{e}_{p} (x, \rho )$, $\theta = e, o$ uniformly converge as $\rho\rightarrow 0$ to an identically zero function. 

\item[(v)] Also, for every $k, 0 < k < m - p - 1$, functions 
$ \frac{d^{k} \phi^{\theta}_{p} (x,\rho )} 
{dx^{k}}$ , $\theta = e, o$ 
also converge uniformly to zero.

\item[(vi)] For $0 \leq k < m - p - 1$, $\theta = e, o$ we have 
\[
\lim_{\rho  \rightarrow \rho_{0} }
\frac{d^{k} \phi^{\theta}_{p} (x, \rho ) }
{dx^{k}} = 
\frac{d^{k} \phi^{\theta}_{p} (x, \rho_{0} ) }
{dx^{k}}, 
\]
where $\rho_{0} \geq 0$, and the limit is uniform with respect to $x$. 
\end{itemize}
\end{theorem}

\begin{proof} From Lemma 5.3 it follows that for $s = i\omega$ 
\begin{equation}
\phi^{e}_{p} (x, \rho ) = \frac{c}{2\pi i}
\int^{i\infty}_{-i\infty}
F (s)\left( s^{p} + ( -s)^{p} \right)e^{s^{2} /(2\rho^{2} ) }
e^{sx} ds .
\end{equation}

Since $m > p + 1$, inequality (61) implies absolute convergence of the integral 
\begin{equation}
f^{ e}_{p} (x) := \frac{c}{2\pi i } \int^{i\infty}_{-i\infty  }
F (s)\left(s^{p} + (-s)^{p}\right)
e^{sx} ds
\end{equation}
for all $x \in  (-\infty , \infty )$. Let us prove that uniformly on $(-\infty , \infty )$, 
\begin{equation}
f ^{e}_{p} (x) = \lim_{\rho \rightarrow \infty} \phi^{e}_{p} (x, \rho ).
\end{equation}
Let  $\epsilon> 0$. Since $m > p + 1$, inequality (61) implies that 
$\exists T > 0$ such that 
 \begin{equation}  
\left| \frac{c}{2\pi i}
\left( \int^{i \infty}_{iT } + \int^{iT}_{i\infty} \right)
F(s)\left(s^{p}+ (-s)^{p} \right) e^{sx} ds
\right| 
< \frac{\epsilon}{2}.
\end{equation}
Let us define $\mu$ by 
\begin{equation}
\mu= sup \left\{ \left|   \frac{c}{2 \pi i} F(s) e^{sx} \left(s^{p}+ (-s)^{p} \right)\right|  : s \in [-iT,T] \right\}
\end{equation}
Then, since $s = i\omega$ there exists $\rho_{0}$ such that (for $\epsilon$ small enough) 
\begin{equation}
\forall\; \rho  \geq \rho_{0} \hspace{1em} \left| 1 - e^{s^{2}/(2 \rho^{2})}\right| < \frac{\epsilon \pi}{2 \mu T c} < 1.
\end{equation}
Formulas (67), (68), (70)-(72) imply that for any $x$ and for any $\rho  \geq \rho_{0}$
\begin{align*}
&|f^{e}_{p}(x) - \phi^{e}_{p}(x,\rho)| \leq 
\left| 
\frac{c}{2\pi i } \int^{i\infty}_{-i\infty  }
F (s)\left(s^{p} + (-s)^{p}\right)
e^{sx} 
\left(1 - e^{s^{2}/(2\rho^{2})}  \right)
ds
 \right|
\\
&\leq 
\left| \frac{c}{2\pi i}
\left( \int^{i \infty}_{iT } + \int^{iT}_{i\infty} \right)
F(s) e^{sx} \left(s^{p} + (-s)^{p} \right)ds \right| 
+ \left| 
\int_{-iT}^{iT}\frac{\epsilon \pi}{2\mu T c} \mu ds
\right| 
\\
&< \frac{\epsilon}{2} + \frac{\epsilon}{2} = \epsilon,
\end{align*}
which proves (69). Therefore, $f ^{e}_{p} (x)$ is continuous. The fact that $f ^{e}_{p} (x)$ 
is $k$ times continuously differentiable follows from (61), the fact that $m > 
p + 1 + k$, and from the theorem on differentiability of improper integrals by 
a parameter (see \cite{k15y}). This proves (i) for $\theta = e$. The proof for (ii) is entirely 
similar to the proof of (i) except that we need to multiply the functions under 
the integral by $sk$ , and we still use (61) to justify absolute convergence.

Let us prove (iii). Since $xf (x)$ is in $L_{1} (\RR$), $F (s)$ is continuously differentiable 
in $s$. Therefore, $F (s)\left( s^{p} + (-s)^{p}  \right)$
is also continuously differentiable in $s$. 
Absolute integrability of $F (s)\left( s^{p} + (-s)^{p}  \right)$ again follows from (61). From (68) 
it follows that 
$f ^{e}_{p} (x) = 
\mathcal{F}^{-1}\left(cF (s)\left( s^{p} + (-s)^{p}  \right)\right).$ Therefore, from the theorem 
on the inverse Fourier transform, it follows that 
$\mathcal{F}\left(f ^{e}_{p} (x)\right) = cF (s) \left( s^{p}+ (-s)^{p}\right)$. This proves (iii) for $\theta = e$. The proof for $\theta = o$ is entirely similar. 

\noindent(iv) Let  > 0. From (67) and the fact that $e^{s^{2}/(2\rho^{ 2} )} < 1$, it follows that 
$\exists T > 0$ such that $\forall\;\rho  > 0$, we have 
 \begin{equation}
 \left| \frac{c}{2\pi i}
\left( \int^{i \infty}_{iT } + \int^{-iT}_{-i\infty} \right)
F(s) e^{s^{2}/(2\rho^{2})}e^{sx} \left(s^{p} + (-s)^{p} \right)ds \right|  < \frac{\epsilon}{2}.
\end{equation}

Let $\mu$ be defined as in (71). Then, 
$\exists \rho_{0} > 0$ such that for any $\rho  < \rho_{0}$, 
\begin{equation} 
\left| e^{s^{2}/(2\rho^{ 2} )} \right|
<  \frac{\epsilon}{4 \mu T}.
\end{equation}
Now (67), (73) and (74) imply that for $\rho  < \rho_{0}$: 
\begin{align*}
&\left|\phi^{e}_{p} (x, \rho ) \right| \leq  
 \left| \frac{c}{2\pi i}
\left( \int^{i \infty}_{iT } + \int^{iT}_{i\infty} \right)
K(s)ds \right| 
+ \left|  \frac{c}{2\pi i} \int_{-iT}^{iT}K(s) ds \right| <
\\
&\frac{\epsilon}{2} +  \frac{\epsilon}{4 \mu T} \left|  \frac{c}{2\pi i} \int_{-iT}^{iT}  F(s)  \left(s^{p} + (-s)^{p} \right)e^{sx}ds \right|
\leq
\\
&\frac{\epsilon}{2} +  \frac{\epsilon}{4 \mu T} 2 \mu T = \epsilon,
\end{align*}
where $K(s) = F(s) e^{s^{2}/(2\rho^{2})}e^{sx} \left(s^{p} + (-s)^{p} \right)$. 
This proves (i) for $\theta = e$. If $\theta = o$ the proof is similar. The proof of (v) is entirely similar. The proof of (vi) is a simplified version of (ii), and is therefore omitted.
\end{proof}
 
The following theorem is important in showing that functions $\phi^{\theta}_{p} (x, \rho )$ 
are essentially localized in a domain with bounded $x$ and $\rho$.

\begin{theorem} If a function $f (x)$ satisfies the conditions of Proposition 5.1 
with $m > p + 1$, then for any $\epsilon  > 0$ there exist $T > 0$ and $Q > 0$ such that 
if 
$|x| > T$ or if $\rho  < Q$, then 
\begin{align}
\left|\phi^{\theta}_{p} (x, \rho )  \right| <  \epsilon,\\ \notag 
\frac{d^{k} \phi^{\theta}_{p} (x, \rho )}{ dx^{k}} < \epsilon . 
\end{align}
\end{theorem}

\begin{proof}
Since $m > p + 1 > 0$ the function $f^{ (k)} (x)$ is in $L_{1} (\RR)$ for $0 \leq k < m - p - 2$. As a result of the properties of convolutions 
$f^{ (k)} (x) * \rho ^{p+1} \Lambda^{ \theta}_{p} (x)$ is in $L_{1} (\RR)$ for these values of $k$. This implies that for every $\rho$
\begin{equation}
 \lim_{|x|\rightarrow \infty  }
\frac{d^{k} \phi^{\theta}_{p} (x, \rho ) }
{dx^{k}} = 0 ,\hspace{2em} \text{ for } 0 \leq k < m - p - 2 .
\end{equation}

Let  $\epsilon> 0$.  From (75) and (76) and Theorem 5.4 (i) it follows that there exist 
positive numbers $R$ and $Q$ such that 
 \begin{equation}
\left|
f^{\theta}_{p} (x) - \phi^{\theta}_{p} (x, \rho )  
\right| 
< \frac{\epsilon}{ 2} , \hspace{1em} \forall \; \rho  > R, \;\;\forall\;x , 
\end{equation}
and 
\begin{equation} 
\left| \phi^{\theta}_{p} (x, \rho  )  \right|
< \epsilon, 
\hspace{1em} \forall \; \rho' <Q, \;\;\forall\;x .
\end{equation}

Since $f^{\theta}_{p} (x)$ is in $L_{1} (\RR)$, $\lim_{|x|\rightarrow \infty} f^{\theta}_{p} (x) = 0$. This implies that  $\exists T_{1}$ such that 
\begin{equation} 
\left| f^{\theta}_{p} (x) \right| 
<  \frac{\epsilon}{2} ,\hspace{1em} \forall\; |x| > T_{1} . 
\end{equation}
Inequalities (77) and (79) imply that 
\begin{equation} 
\left| \phi^{\theta}_{p} (x, \rho ) \right| 
< , \epsilon \hspace{1em} \forall \; \rho  > R, \;\; \forall \; |x| > T_{1} .
\end{equation}
Using (80) and Theorem 5.4 (vi) we can show that for every $\rho  \in  [Q, R]$ there 
exists an open neighborhood $O_{y}$ such that for some positive number $T (\rho )$ 
 \begin{equation}
\left|\phi^{\theta}_{p} (x, \rho ) \right| 
< \epsilon ,  \hspace{1em} \forall \; \rho'  \in O_{y}, \;\text{and}\; |x| > T(\rho) .
\end{equation}
Since $[Q, R]$ is compact, there exists a finite number of open sets $O_{\rho_{1} }, . . . , O_{\rho _{n} }$, 
covering $[Q, R]$. Let $T = \max( T_{1} , T (\rho_{ i }) , i = 1, . . . , n)$, then (78), (80) and 
(81) imply that 
 \begin{equation*}
\phi^{\theta}_{p} (x, \rho )  
< \epsilon \hspace{1em} \text{if either } 
|x| > T \text{ or } \rho  < Q .
\end{equation*}
 
The second condition in (75) is proved similarly. 
\end{proof}

\noindent \textbf{Remark}: The results of this section can be generalized to functions $f (x)$ 
such that they and their derivatives have a finite number of jump discontinuities. 
\begin{theorem}
Let $f (x)$ be as in Proposition 5.1 with $m > p + 1$. Then if 
$\Delta > 0$ is such that $p - \Delta > 0$, $p + \Delta + 1 < m$, then all convergence results 
in Theorem 5.4 are uniform with respect to $q \in  [ p - \Delta, p + \Delta ]$. In other 
words, $\forall\; \epsilon > 0$, $\exists \rho_{0}$ such that if $\rho  > \rho_{0}$ , then 
$\left|
\phi^{\theta}_{q} (x, \rho ) - f^{\theta}_{q} (x) 
\right| < \epsilon$, $\forall\;x$ and 
$\forall\;q \in  [ p - \Delta, p + \Delta ]$ (and similarly for convergences $\phi^{\theta}_{q} (x, \rho ) 
\rightarrow \phi^{\theta}_{q} (x, \rho_{1} )$
as $\rho  \rightarrow \rho_{1}$ and $\phi^{\theta}_{q} (x, \rho ) \rightarrow 0$ uniformly as 
$\rho \rightarrow 0$ for $q \in  [ p - \Delta, p + \Delta ]$. 
Furthermore, for any  $\epsilon > 0$, 
$\exists T > 0$ and $Q > O$ such that if $|x| > T$ or if  $\rho  < Q$, then 
$\forall\; q \in  [p - \Delta, p + \Delta]$ where $p - \Delta > 0, p + \Delta + 1 < m$ we have 
\begin{equation}
\left|\phi_{q} (x, \rho )\right| <  \epsilon 
 \end{equation}
\begin{equation} 
\left|
\frac{d^{k} \phi^{\theta}_{q} (x, \rho )}{dx^{k}}  
\right| 
< \epsilon \hspace{1em} \text{for } 0 < k < m - p - 2 .
\end{equation}
\end{theorem}
\begin{proof}
The proof is entirely similar to the proof of Theorems 5.4 and 5.5, except that all inequalities can be satisfied uniformly on $[p - \Delta, p + \Delta]$. 
\end{proof}

\begin{theorem}
Let $f (x)$ be as in Proposition 5.1 with $m > p + 1 + 2l$ with $l > 0$. Then there is a continuous function $\Psi
^{\theta}_{p} (x, \sigma)$, $\theta = e, o$, defined for all positive $\sigma$ such that: 
\begin{itemize}
\item[(i)] $\Psi
^{\theta}_{p} (x, \sigma) = \phi^{\theta}_{p} (x, \rho )$ for $\rho  > 0, \rho  = \frac{1}{\sigma} $, and $\Psi
^{\theta}_{p} (x, 0) = f^{\theta}_{p} (x)$.
 
\item[(ii)] Function $\Psi
^{\theta}_{p} (x, \sigma)$ has continuous mixed derivatives 
$\frac{ \partial^{k+n}}{\partial x^{k} \partial  \sigma^{n}} \Psi
^{\theta}_{p} (x, \sigma)$
for any $k$ and $n$ such that $2n + k \leq 2l$.
 
\item[(iii)] $\Psi
^{e}_{p} (x, \sigma)$ is even, and$\Psi
^{o}_{p} (x, \sigma)$  is odd. 

\item[(iv)] $\Psi
^{\theta}_{p} (x, \sigma) \sim F (s) \left( s^{p}+ (-s)^{p} \right) e^{s^{2}\sigma^{2}/2} \hspace{1em}$ if  $\theta = e$, and 

$\Psi
^{\theta}_{p} (x, \sigma) \sim F (s) \left( s^{p}- (-s)^{p} \right) e^{s^{2}\sigma^{2}/2} \hspace{1em}$ if  $\theta = e$,if $\theta = o$

\item[(v)] $\sigma \frac{\partial^{2} \Psi
^{\theta}_{p} (x, \sigma)}
{\partial  x^{2}}
 = 
\frac{\partial\Psi
^{\theta}_{p} (x, \sigma)}{\partial \sigma}$ for every $x$ and $\sigma$. 
\end{itemize}
\end{theorem}
\begin{proof} Let $\theta = e$. Then we define $\Psi
^{\theta}_{p} (x, \sigma)$ by 
\begin{equation}
\Psi
^{\theta}_{p} (x, \sigma) = \frac{c}{2\pi i}
\int^{i \infty }_{-i\infty} F (s)\left( s^{p}+ (-s)^{p} \right) e^{s^{2} \sigma^{2} /2} e^{sx} ds
,\end{equation}
where $F (s) = \mathcal{F} (f (x))$, $s = i\omega$.
As in the proof of Theorem 5.4, one can show that $F (s)\left( s^{p}+ (-s)^{p} \right) e^{s^{2} \sigma^{2} /2}$ is indeed the Fourier transform of $\Psi
^{\theta}_{p} (x, \sigma)$ 
for $\theta = e$. This proves (iv). (iii) is obvious. (i) follows from Theorem 5.4. 

Let us prove (ii). From the theorem on differentiation of improper integrals by a parameter (see \cite{k15y}, 11.55 a), we can interchange differentiation of the right hand side of (85) with integration if the derivative of the expression under the integral absolutely converges. However, 
\begin{equation}
\frac{\partial^{k+n}}{\partial  x^{k} \partial  \sigma^{n}}
 F (s)\left( s^{p}+ (-s)^{p} \right) e^{s^{2} \sigma^{2} /2} 
 e^{sx} 
 = s^{k} p(s, \sigma)e^{s^{2} \sigma^{2} /2} e^{sx} 
 \left(F (s) \left(s^{p}+ (-s)^{p}\right)\right)
\end{equation}
where $p(s, \sigma)$ is a polynomial whose degree on $s$ is less than or equal to $2n$. 

Now (85) together with the inequality $m > p + 1 + 2l$ imply that if 
$2n + k \leq 2l$, the integral in (84) absolutely converges. This allows us to 
interchange integration and differentiation. Since the derivatives of all orders 
of the expression under the integral in (84) exist, (ii) follows. The proof for 
$\theta = o$ is entirely similar. In order to prove (v), we simply differentiate 
(84). 
\end{proof}

\begin{theorem} Let $f (x)$ be as in Proposition 5.1 with $m > p + 1$ and let 
$I = [a, p]$  ($a > 0$ if $\theta = o$ and $a \geq 0$ if $\theta = e$) be an interval on the real line. 
Then the set $D_{\epsilon}$ defined by condition 
\begin{equation}
D_{\epsilon} = 
\left\{
x, \sigma, p_{1} : 
\left| 
\Psi^{\theta}_{p_1} (x, \sigma) 
\right|
 \geq 
\epsilon, p_{1} 
\in  [0, p]
\right\}
\end{equation}
is compact if $\epsilon > 0$.
\end{theorem}

\begin{proof} From Theorems 5.7 it follows that 
$\forall\; p_{1} \in  I$ , $\exists\Delta > 0$ such that the set 
\begin{equation}
D_{\epsilon}(p_{1} ) = 
\left\{
x, \sigma, p_{2} \;:\; p_{2} 
\in  [p_{1} - \Delta, p_{1} + \Delta], \;\;
\left| \Psi^{\theta}_{p_1} (x, \sigma)  \right|  \geq \epsilon
  \right\}
\end{equation} 
is bounded if  $\epsilon> 0$. Since $M$ is compact, there is a finite covering of $M$ with 
intervals $[q_{j} - \Delta_{j} , q_{j} + \Delta_{j} ]$ such that each $D_{\epsilon}$ satisfying (86) for $p_{1} \in  \{q_{j}\}$ and $\Delta = \Delta_{j}$ is bounded; thus, $D_{\epsilon}$ is also bounded as a union of a finite number of bounded sets. Since $D_{\epsilon}$ is closed in $\RR^{3} = \{x, \sigma, p\}$, it follows that 
$D_{\epsilon}$ is compact. 
\end{proof}

\begin{theorem}
If in Theorem 5.8 $m > p + 1 + k$, then the set 
\begin{equation}
D_{\epsilon} = 
\left\{
x, \sigma, p_{1} : 
\left| \frac{\partial^{k}\Psi^{\theta}_{p_1} }{\partial x^{k}} (x, \sigma) 
\right|
 \geq 
\epsilon, p_{1} 
\in  [0, p]
\right\}
\end{equation}
is compact if  $\epsilon> 0$. 
\end{theorem}

\begin{proof}
Let us take 
$\frac{d^{k} f}{ dx^{k}}$ instead of function $f (x)$ in Theorem 5.8. Then the 
function defined by the right side of (84) where 
$F (s) = \mathcal{F}\left(\frac{d^{k} f}{dx^{k}}\right)$
is actually $\frac{\partial^{k}\Psi^{\theta}_{p} }{\partial x^{k}} (x, \sigma) $.
Thus, Theorem 5.9 follows from Theorem 5.8. 
\end{proof}

\begin{theorem}
(i) Consider functions $\Lambda^{\theta}_{p} (x)$, $\theta = e, o$, as functions of $p$ 
given by equalities (7) and (8) where $p$ is now an arbitrary complex number. 
Then $\Lambda^{\theta}_{p} (x)$ is an analytic, even entire function of two variables $x$ and $p$. 
(ii) Let function $\Psi^{\theta}(p, x, \sigma) := \Psi^{\theta}_{p} (x, \sigma)$ be defined by (84) where $F (s)$ is the Fourier transform of $f (x)$. For $f (x)$ as in Proposition 5.1 with $m > p + 1$, 
$\Psi^{\theta}(p, x, \sigma)$ is a continuous function of three variables. It has continuous 
mixed order derivatives 
$\frac{\partial^{k+n+l }\Psi^{\theta}(p, x, \sigma)}{\partial  x^{k} \partial  \sigma^{n} \partial  p^{l}}$ where $l$ is arbitrary, and $k$ and $n$ satisfy conditions of Theorem 5.7. 
\end{theorem}

\begin{proof}
Let us assume $\theta = e$. The ratio test shows that the
series (7) 
converges absolutely and uniformly by $p$ and $x$ on every bounded set of pairs 
$(x, p), |x| < R_{1}$, $|p| < R_{2}$ . Since each term
in (7) 
 is an entire function, the sum is also entire. This proves (i). The proof of (ii) is similar to that of (ii) 
of Theorem 5.7 and is therefore omitted. 
\end{proof}
\noindent \textbf{Remark}: Since $\phi^{\theta}_{p} (x, \rho ) = \rho p+1 \Lambda^{\theta}_{p}(x\rho )*f (x)$, it is easy to show that if the $n$th derivatives of $f (x)$ are in $L_{1} (\RR)$ for $n \leq m$, then 
$\frac{\partial^{  n} \phi^{\theta}_{p}}{dx^{n}} (x, \rho ) = 
\rho^{ p+1} \Lambda^{\theta}_{p} (x\rho ) 
* 
\frac{\partial  ^{n} f (x) }{dx^{n}}$
 and therefore 
$\frac{\partial^{n} \Psi^{\theta}_{p}}{dx^{n}} (x, \sigma) = 
\frac{1}{\rho ^{p+1}} \Lambda^{\theta}_{p}( \frac{x}{\sigma} ) 
*
\frac{\partial  ^{n} f (x) }{dx^{n}}$ for $\sigma > 0$. 
This shows that if a certain statement is true about contours $\Psi^{\theta}_{p} (x, \sigma) = c$, 
then it is true for contours 
$\frac{\partial^{n} \Psi^{\theta}_{p} (x, \sigma)}{dx^{n}} = c$ if the assumptions for 
$\frac{\partial^{n} f (x) }{dx^{n}}$ in 
the latter case are the same as the assumptions for $f (x)$ in the former.

\section{Experimental Results}
The purpose of our experiments was to demonstrate the following important 
properties of kernels $g_{p} (x, y)$ (we use $\pi  = e$). 

\noindent \textbf{1)} We show that the trees resulting from convolutions with kernels with 
fractional (not integral) $p$ can distinguish between signals that are indistinguishable by trees resulting from convolutions with integer indexed kernels 
(at least for low integers). Our experiments show that two signals which are 
very close to each other and whose higher order derivatives are also close can 
be distinguished by a convolution with $g_{p}$ with a low $p$. Since higher $p$ results 
in significantly higher noise, this feature is highly desirable computationally 
and shows the advantages of using kernels with fractional $p$ over Gaussian 
and its derivatives. Another advantage is that of constructing a surface for 
zero or level crossings and classifying it topologically with resulting integer 
invariants (the number of handles), which is discussed below. 

Figures 1A and 1B show two signals $y_{1}$ and $y_{2}$ . Their respective convolutions for $p = 0, 1, 2$ are shown in Figs.2A, 2B, 3A, 3B, 4A, and 4B respectively. Obviously, both Witkin and topological trees are indistinguishable. 
However, for fractional $p = 1.35$, we get different trees (see Figs.5A and 5B). 
In order to save space, we do not always draw the trees, assuming the reader 
can reconstruct it directly from the crossing curves. 

\noindent\textbf{2)} We demonstrate that topological trees may be less sensitive to noise than 
Witkin trees. 

In Fig.6A and 6B we show the original signal and the same signal with 
noise added. In Fig.7A and 7B we show the corresponding crossing curves. 
One can see that the crossing of 7B experiences Witkin bifurcation resulting 
in a different tree from 7A; however, topological trees are the same. 

\noindent\textbf{3)} The following set of experiments demonstrates that 

\noindent(a) the trees resulting from fractional convolutions are relatively insensitive 
to noise (even though topological trees are more insensitive than the Witkin 
ones); 

\noindent(b) local modification of signals results in different trees allowing us to dis- 
tinguish between the original and modified signal; 

\noindent(c) bifurcations occur frequently enough to provide additional fractional
trees for low p, but not so frequently as to make computations difficult. 
Fig.8A shows a noiseless contour of England (256 samples). 

\noindent(d) The contours best and the least sensitive to noise tend to be for $p < 2$. 
The trees corresponding to the $x$ coordinate were found for various $p$ between 0.5 and 2.15. Here we only show contours at and around bifurcation points. Fig.9A and 9B corresponds to $p$ = 0.6 and $p = 0.7$ respectively, 
Fig.10A and 10B to $p$ = 0.8 (bifurcation point) and $p$ = 0.8 respectively, 
Fig.11A and 11B to $p$ = 1.02 (bifurcation point) and $p$ = 1.04 respectively. 
The contours of the noisy map of England are shown in Figs. 12A, 12B, 
13A and 13B for $p$ = 0.6, $p$ = 0.7, $p$ = 0.78, and 0.79 (bifurcation point) and 
Figs.14A, 14B, 15A and 15B corresponding to $p$ = 0.8, $p$ = 0.9 (bifurcation 
point) $p$ = 0.95, $p$ = 0.97 respectively.

As the results of simulations have shown, the trees in the region $p \in  [0.5, 1.5]$ 
are the same as for the noiseless map, but the bifurcations appear a little 
earlier. From these results, it is also clear that topological trees are easier 
to construct from the given contour. 
We have repeated this experiment by locally modifying the map (chopping off the 
southwestern part of England's map). See Fig.16 for the noiseless 
and Fig.17 for the noisy maps. For noiseless signals, we have contours in 
Figs.18A, 18B, 19A and 19B for $p$ = 0.6, $p$ = 0.7 (bifurcation point), $p$ = 0.8, 
p = 0.9 and in Figs.20A, 20B, 21A and 21B for $p$ = 1, $p$ = 1.1, $p$ = 1.4, and 
p = 1.5 (bifurcation point) and in Figs.22A, 22B are contours for $p$ = 1.6 
and $p$ = 1.7 respectively. We can see that the resulting trees are different 
from the previous unchopped contour of England. We only show the noisy 
maps for chopped off England for $p$ = 0.72, and $p$ = 0.74 in Figs.23A, 23B 
respectively. The bifurcation points are slightly delayed at the presence of 
noise and occur at $p$ = 0.72 and $p$ = 1.7 (not shown). Experiments just 
described were conducted for other signals as well. We chose level crossings 
at values of $c (\pi (x, \rho ) = c)$ close to 0 to avoid contours which do not close 
up. 

\section{Conclusion}
Below we discuss important conceptual contributions of this paper as well 
as a practical significance of our results. With the assumption of monotonicity and scaling property we have constructed a maximal set of kernels. 
This allows us to construct maximal sets of tree invariants (i.e., trees which 
are invariant with respect to shape preserving signal transformations) for 
classification, recognition and feature extraction of signals. 

\noindent(a) In Theorem 4.1 we characterized the maximal set of kernels which satisfy 
our assumptions (axioms) for a linear scale-space, including the assumptions 
of scaling invariance, strong monotonicity, and stability. 

\noindent(b) Fourier transforms of our kernels were derived in Theorem 3.3 for quick 
and efficient computations of convolutions. 

\noindent(c) We have introduced a new topological method of tree construction which 
we believe is much less sensitive to noise than WitkinÕs method. 

\noindent(d) For the purposes of applications, we demonstrated the usefulness of 
introducing kernels for fractional $p$ in practical problems by showing that the 
Gaussian and its derivatives fail to distinguish between two similar signals 
whereas the kernels for fractional $p$ succeed in this task. Furthermore, we 
have observed that our scale-space method is relatively immune to noise for 
low $p$ as well as capable to detect a local modification of signals. In addition 
our methods allow us to
\begin{itemize}
\item[(*)] distinguish between the signals indistinguishable by a Gaussian ker- 
nel; 
\item[(*)] use the topological trees (along with WitkinÕs trees) which are less 
noise sensitive; and 
\item[(*)] use level curves which assure that the contours will close up so that 
the tree can be constructed. 
\end{itemize}

\noindent e) We also gave a rigorous treatment of the structure of level crossings 
of functions $\frac{\partial^{k}\Psi (x,\sigma,p)}{ \partial  x^{k}}$ which, in particular, generalizes many assumptions 
about Gaussian convolutions (most of them have never been proved before). 
Thus, we have given a solid foundation to many papers \cite{k2y, k20y, k21y, k11y, 
k4y} for the particular case of the Gaussian convolutions $(\Lambda^{e}_{p} * f )$ as a family 
of kernel function. 

We, therefore, believe that our method will significantly enhance the 
applications of the linear scale-space method.


%

\end{document}